\documentclass[11pt,A4paper,centertags,reqno]{amsart}%

\usepackage{amssymb}
\usepackage{amsmath}%
\usepackage{amsthm}
\usepackage{braket}
\usepackage{mathtools}
\usepackage{mathrsfs}
\usepackage{bbm}
\usepackage{xcolor}
\usepackage{amsfonts}%
\usepackage{color,graphicx}
\usepackage[utf8]{inputenc}
\usepackage[T1]{fontenc}
\usepackage[english]{babel}
\usepackage{hyperref}
\usepackage{latexsym}
\usepackage{fancyhdr}
\usepackage{epsfig}
\usepackage{float} 
\usepackage{tikz}
\usepackage{etex}
\usepackage[foot]{amsaddr}
\definecolor{refkey}{rgb}{1,1,1}
\definecolor{labelkey}{rgb}{0,0,1}

\DeclareMathAlphabet{\mathpzc}{OT1}{pzc}{m}{it}
\DeclareMathOperator{\Fix}{Fix} 

\def\e{\varepsilon}
\def\d{\displaystyle}
\newcommand{\mint}{- \kern -,375cm\displaystyle\int}
\def\deb{\rightharpoonup}
\def\ov{\overline}
\def\md{\medskip}
\def\sm{\smallskip}
\def\bg{\bigskip}
\def\K{\mathcal K}
\def\P{\mathcal P}
\def\C{\mathcal C}
\def\D{\mathcal D}

\def\B{\mathcal B}

\def\f{\varphi}

\def\l{\left}

\def\ri{\rightarrow}
\def\r{\right}
\def\np{\noindent}

\def\beq{\begin{eqnarray}}
\def\eeq{\end{eqnarray}}
\def\dis{\displaystyle}
\def\widetilde{\mathaccent"0365}

\def\erre{\mathbb{R}}

\newcommand{\enne}{\mathbb{N}}

\newcommand{\T}{\mathpzc T}

\newtheorem{theorem}{Theorem}[section]

\newtheorem{definition}[theorem]{Definition}

\newtheorem{lemma}[theorem]{Lemma}

\newtheorem{proposition}[theorem]{Proposition}

\numberwithin{equation}{section}

\mathtoolsset{showonlyrefs}

\begin{document}
\author[I.~Benedetti]{Irene Benedetti}
\address{Irene Benedetti, Department of Mathematics and Computer Science,
 University of Perugia, Via Luigi Vanvitelli, 1, I-06123 Perugia, Italy.}
\email{irene.benedetti@unipg.it}
 \author[P.~Rubbioni]{Paola Rubbioni}
\address{Paola Rubbioni, Department of Mathematics and Computer Science,
 University of Perugia, Via Luigi Vanvitelli, 1, I-06123 Perugia, Italy.}
 \email{paola.rubbioni@unipg.it}

\title[]{Impulsive delay differential inclusions  applied to optimization problems}

\date{}

 \begin{abstract}
 We study a class of semilinear impulsive differential inclusions with infinite delay in Banach spaces. The model incorporates multivalued nonlinearities, impulsive effects, and infinite memory, allowing for the description of systems influenced by long-lasting past states and sudden changes. We prove the existence of mild solutions and the compactness of the solution set using fixed point methods and measures of noncompactness. The theoretical results are applied to an abstract optimization problem and to a population dynamics model.
\end{abstract}

\maketitle



\section{Introduction}

\noindent In the present paper we consider the following class of semilinear impulsive differential inclusions with infinite delay in a Banach space $E$
$$
(P) \left\{
\begin{array}{l} y'(t)\in A(t)y(t)+ F(t,y(t),y_t)\;, \;\mbox{
a.e. } t\in \, [t_0,T]\ ,\, t\neq
t_k, \, k=1,\dots,m\\
\\
y(t)=\varphi^*(t)\;,\; t\le t_0\\
\\
y(t_k^+)=y(t_k)+I_k(y_{t_k})\;,\;k=1,\dots,m 
\end{array}\right.
$$
where $\{A(t)\}_{t\in [t_0,T]}$ 
%\\
%\purple{otteniamo quindi una famiglia $(T(t,s))$ con (t,s) non in $\Delta $ ma in $\Delta_{t_0}=\{t_0\le s\le t\le T\}$}
%\\
is a family of linear operators defined on the same domain $D(A)$; 
$F:[t_0,T] \times E\times \B \to \P(E)$ is a given multivalued map; 
$\B$ is a phase space that will be defined later; 
$y_t(\theta)=y(t+\theta),\,\theta \in(-\infty,0]$, belongs to $\B$ and represents the history of the state up to the time $t\in [t_0,T]$; $t_0<t_1<\cdots<t_m<t_{m+1}=T$ denote the impulse times, associated with impulse functions $I_k:\B\to E$, $k=1,\cdots, m$; $\varphi^*:\,(-\infty,t_0]\to E$ is a given function such that $\varphi^*_{t_0}\in \B$, and represents the initial state of the system; $y(t^+):=\lim_{s\to t^+}y(s)$.

\noindent Differential equations with infinite delay, also referred to as functional differential equations (FDEs) or delay differential equations (DDEs) with infinite delay, model systems whose evolution at a given time depends not only on the present state but also on the entire past history extending to minus infinity. Such equations are particularly relevant in applications where past states exert long-lasting effects. For instance, in ecosystems with extremely slow resource regeneration or in environments affected by long-lived, bioaccumulative pollutants, assuming a finite delay may neglect the persistent influence of historical resource depletion or past contamination on current population dynamics. Infinite-delay models therefore provide a more faithful mathematical representation of systems in which past effects do not fade within any finite time horizon.

\noindent For the above reasons, starting from the pioneering work of Hale and Kato \cite{hk}, equations with infinite delay have been the subject of extensive research over the past decades. Comprehensive treatments of this theory can be found in the monographs \cite{hmn} and \cite{LWZ}. Further significant contributions have been made by Henriquez and collaborators; see, for example, \cite{H1,H2,H3}. In \cite{gorr02}, the authors investigate semilinear functional differential inclusions in a Banach space with infinite delay on an abstract phase space, while in \cite{Fan2008} these results are extended to the fully nonlinear setting. The topological structure of the solution set for abstract inclusions with infinite delay, covering both semilinear and quasilinear cases, is studied in \cite{Guedda}. More recently, the problem of viability for semilinear differential equations with infinite delay has been addressed in \cite{EzSt25}.

\noindent Another relevant line of research concerns differential equations with infinite delay subject to impulsive effects. Such models account for dynamics in which the solution may exhibit discontinuities in response to sudden changes in the system. Applications include, for instance, financial models capturing abrupt market shocks or agricultural systems affected by the sudden introduction of chemical agents or extreme atmospheric events. The existence of mild solutions for impulsive semilinear functional differential inclusions with state-dependent delay and multivalued jumps in Banach spaces was established in \cite{BZ}. Related recent results can be found in \cite{Benchora2024}, where the existence of periodic mild solutions for a class of impulsive integro-differential inclusions with infinite delay is studied using fixed point techniques combined with measures of noncompactness and the resolvent operator. The existence of mild solutions for impulsive partial neutral functional differential equations of first and second order is addressed in \cite{Hernandez2007}. We also refer the reader to the monograph \cite{gho}, which provides a comprehensive treatment of differential inclusions with infinite delay and impulsive effects in a finite-dimensional framework. Still within a finite-dimensional setting, more recently, existence and uniqueness results for retarded equations with non-instantaneous impulses, nonlocal conditions, and infinite delay have been established in \cite{Leiva2023}. 

\noindent In conclusion, problem $(P)$ is formulated in a highly general setting and is applicable to a wide range of real-world phenomena. It incorporates, within a unified framework, a multivalued term, allowing for the presence of measurement uncertainties, impulsive effects, which model sudden environmental changes, and an infinite delay, accounting for situations in which past events, even in the distant past, influence the current state. In this paper, we establish the existence of at least one mild solution (see Definition \ref{def-mild-sol}) and the compactness of the solution set, and subsequently apply these results to an abstract optimization problem. The work concludes with an application to the study of optimal solutions in population dynamics. The results are obtained using fixed point techniques combined with regularity assumptions expressed in terms of measures of noncompactness. A closely related problem, addressed using similar methods, was analyzed in the previously cited paper \cite{BZ}, where, in contrast to the present work, the nonlinear term depends only on the delay. Here, we allow the nonlinearity to depend also on the state function and prove the existence of a solution without imposing any assumptions on the impulse functions. 

\noindent The paper is organized as follows: in Section \ref{Preliminaries} we recall some known definitions and results useful in the sequel of the manuscript; in Section \ref{phase-space} we describe the phase space that we consider both in the abstract setting and in the application; in Section \ref{s-P} are collected the main assumptions required on the problem $(P)$, the definition of the solution and some preliminary results needed to prove the main results; in Section \ref{main} the non emptyness and compactness of the solution set is proven; finally, in Section \ref{s-A} an application of the theoretical results to a model in population dynamics is presented.

\section{Preliminaries}
\label{Preliminaries}

Let $X$, $Y$, be two topological vector spaces. We denote by $\P(Y)$ the family of all non-empty subsets
of $Y$ and define
$$
\K(Y)=\l\{C\in\P(Y),\,\mbox{compact}\r\}\ ; \qquad \K v(Y)=\l\{D\in\P(Y),\,\mbox{compact and convex}\r\}\ .
$$

\noindent A multivalued map $F:X\ri \P(Y)$ is said to be:
\begin{itemize}
\item[-] {\em upper semicontinuous}  (for shortness u.s.c.) if $F^{-1}(V)=\l\{x\in X:F(x)\subset V\r\}$ is an open subset of $X$ for every open $V\subseteq Y$;
\item[-] {\em closed}  if its graph $G_F=\l\{(x,y)\in X\times Y:\,y\in F(x)\r\}$ is a closed subset of $X\times Y$.
\end{itemize}

\noindent For u.s.c. multimaps the following result holds (see, e.g.\cite{g}).
\begin{proposition}
\label{usc2} 
Let $F:X\ri\K(Y)$ be an u.s.c. multimap. If $C\subset X$ is a compact set then its image $F(C)$ is a compact subset of $Y$.
\end{proposition}
\bg

Let $E$ be a real Banach space.

\np If $(N,\geq)$ is a partially ordered set, we recall that a map $\beta:\P(E)\ri N$ is said to be a {\em measure of non
compactness} (MNC) in $E$ if
\begin{equation*}
\beta(\ov {co}\Omega)=\beta(\Omega)
\end{equation*}
for every $\Omega\in \P(E)$ (see, e.g. \cite{KOZ} for details).

\np A measure of non compactness $\beta$ is called:
\begin{itemize}
\item[(i)] monotone if $\Omega_0,\,\Omega_1\in\P(E)$, $\Omega_0
\subseteq \Omega_1$ imply
  $\beta(\Omega_1)\geq\beta(\Omega_0)$;
\item[(ii)] nonsingular if
$\beta(\l\{c\r\}\cup\Omega)=\beta(\Omega)$
  for every $c\in E$, $\Omega\in \P(E)$;
\item[(iii)] real if $N = [0,+\infty]$ with the natural ordering
and $\beta(\Omega)
 < +\infty$ for every bounded $\Omega$;
\item[(iv)] regular if $\beta(\Omega)=0$ is equivalent to the
relative compactness of $\Omega$.
\end{itemize}

\np A well known example of measure of non compactness satisfying all of the above properties is the Hausdorff MNC
\begin{equation*}
\chi(\Omega)=\d\inf\l\{\e>0:\Omega\,\mbox{has a finite
$\e$-net}\r\}.
\end{equation*}

If $X$ is a subset of $E$ and $\Lambda$ a space of parameters, a multimap $F:X\ri\K(E)$, or a family of multimaps $G:\Lambda\times X\ri\K(E)$, is called {\em condensing} relative to a MNC $\beta$, or $\beta$-condensing, if for every $\Omega\subseteq X$ that is not relatively compact we have, respectively
\begin{equation*}
\beta(F(\Omega))\ngeq \beta(\Omega)\ \mbox{ or }\
\beta(G(\Lambda\times\Omega))\ngeq\beta(\Omega)\ .
\end{equation*}

The proof of the main result is based on the following Fixed Point Theorem.
\begin{theorem}[\cite{KOZ}, Corollary 3.3.1]
\label{fixedpoint}
Let $M$ be a convex and closed subset of $E$, $F:M\ri\K v(M)$ be a closed, $\beta$-condensing multimap, with $\beta$ a nonsingular MNC defined on subsets of $E$, then $Fix\, F \neq \emptyset$.
\end{theorem}

The following property of the fixed points set of $F$ will be useful in the sequel.

\begin{proposition}[\cite{KOZ}, Proposition 3.5.1]
\label{fixedpoint}
Let $M$ be a closed subset of $E$, $F:M\ri\K(E)$ a closed multimap, which is $\beta$-condensing on every bounded subset of $M$, with $\beta$ a monotone MNC defined on subsets of $E$. If $Fix\, F=\{x\in M:\,x\in F(x)\}$ is bounded, then it is compact.
\end{proposition}
\sm

Let $[a,b]$ be an interval of the real line. By the symbol ~$L^1([a,b];E)$~ we denote the space of all Bochner summable functions and, for simplicity of notations, we write ~$L^1_+([a,b])$~ instead of $L^1([a,b];\erre^+)$.\sm

\np A multifunction ~${G}: [a,b] \to \K(E)$~ is said to be
\begin{itemize}
\item[-] {\em integrable}\, if it has a summable selection ~$g \in L^1([a,b];E)$; 
\item[-] {\em integrably bounded}\, if there exists a summable function ~$\omega(\cdot) \in L^1_+([a,b])$~ such that
$$
\|{G}(t)\| := \sup \{\|g\| ~:~ g \in {G}(t)\} \leq
\omega (t), \quad \mbox{ a.e. } t\in [a,b].
$$
\end{itemize}

Finally, a countable set ~$\{f_n\}_{n=1}^{\infty} \subset L^1([a,b];E)$~ is said to be {\em semicompact} if:
\begin{itemize}
\item[(i)] it is integrably bounded: ~$\|f_n(t)\| \leq \omega(t)$~for a.e. $t \in [a,b]$ and every ~$n \geq 1$~, where $\omega(\cdot) \in L^1_+([a,b])$; 
\item[(ii)] the set~$\{f_n(t)\}_{n=1}^{+\infty}$~ is relatively compact for a.e. $t\in [a,b].$
\end{itemize}
\begin{theorem}[\cite{Diestel}, Corollary 2.6]
            \label{wcomp}
            Let
            $
            \mathcal{F} \subset L^1([a,b];E)
            $
            be an integrable bounded family of functions such that $\mathcal{F}(t)=\left \{ f(t), f \in D \right \}$ is weakly relatively compact for a.e. $t \in [a,b].$
            Then $\mathcal{F}$ is weakly relatively compact in $L^1([a,b];E)$.
            \end{theorem}
\begin{proposition}[\cite{CR-2005}, Lemma 2]
\label{p-t}
Let $\{U(t,s)\}$ be an evolution system on $E$ and $G:L^1([a,b];E) \to C([a,b];E)$ be the operator defined as
\begin{equation}
    \label{G}
    G(f)(t)=\displaystyle\int_a^b U(t,s)f(s)\,ds.
\end{equation}
Then, for every semicompact sequence $\{f_n\}_{n=1}^{\infty}$ the sequence $\{Gf_n\}_{n=1}^{\infty}$ is relatively compact and moreover if $f_n \rightharpoonup f_0$, then $G(f_n) \to G(f_0)$.
\end{proposition}
The Heine Cantor Theorem is of a key importance in order to prove the existence of a integrable selection of the multivalued map $F$, thus bearing in mind that the space $\B$ is not a metric space, we recall the following generalization of it for non metric spaces.
\begin{theorem}[see e.g. \cite{W}, Corollary 36.20] 
\label{uniform-cont}
Every continuous function on a compact $T_2$ space with values in a pseudometric space is uniformly continuous.
\end{theorem}

\noindent Let $[t_0,T]$ be a fixed interval, $t_0\ge 0$. Let $\Delta =\{(t,s)\in [t_0,T]\times [t_0,T]\, :\, t_0\le s\le t\le T\}$, we recall (see, e.g. \cite{p}) that a two parameter family $\{U(t,s)\}_{(t,s)\in\Delta}$, $U(t,s):E\to E$ bounded linear
operator, $(t,s)\in\Delta$, is called an {\em evolution system} if the following conditions are satisfied:
\begin{enumerate}
\item $U(s,s)=I,\ 0\le s\le T$\, ; \quad $U(t,r)U(r,s)=U(t,s)$,\ $t_0\le s\le r\le t\leq T$\, ; \item $(t,s)\mapsto U(t,s)$\, is strongly continuous on $\Delta$ (see, e.g. \cite{k}).
\end{enumerate}
For every evolution system, we can consider the correspondent {\em evolution operator} ~$U:\Delta\to { L}(E)$, where ${L}(E)$ is the space of all bounded linear operators in $E$.

\np We observe that, since the evolution operator $U$ is strongly continuous on the compact set $\Delta$, there exists a constant $D=D_\Delta>0$ such that \beq \label{D} \|U(t,s)\|_{{L}(E)}\le D\, , \quad (t,s)\in\Delta\ . \eeq

\section{The generalized phase space}
\label{phase-space}

The phase space for equations with infinite delay was introduced by Hale and Kato in \cite{hk}, by means of a set of axioms. It is a suitable framework for describing situations in which the past state of the system influences its present state. The possibility of incorporating an infinite delay allows one to control the “amount of past” that one needs or wishes to take into account to obtain solutions that best correspond to the problem under consideration. We refer to the monograph \cite{hmn} for a detailed discussion of phase space and of examples of spaces that satisfy its defining axioms. 

It should be noted, however, that the phase spaces described in that text and normally used in the literature are not suitable for the impulsive case, where solutions are not necessarily continuous but piecewise continuous, with jump discontinuities and left continuity. For this reason, in the impulsive case the axioms of Hale and Kato need to be slightly revisited. For functions taking values in $\mathbb{R}^n$, see the book of Graef, Henderson, and Ouahab \cite{gho}. We will follow their approach and introduce the following function set.

\begin{definition}
\label{d-B}
 Let ${\mathpzc T}=\{t_0,\dots,t_m\} \subset [0,T]$ be a set of fixed real numbers with $t_0<t_1<\cdots <t_m$, and $E$ a  Banach space. A {\em generalized phase space} $\B$ associated to ${\mathpzc T}$ is a linear space, with a seminorm $\|\cdot\|_\B$, of functions mapping $(-\infty,0]$ into $E$, satisfying the following axiom:
  
\begin{itemize}
\item[(B)] if $y:\,(-\infty, T]\to E$ is such that, for every $k=1\dots, m$, the restrictions $y_{|\, ]t_{k-1},t_k]}$ are continuous and there exist $y\big(t_k^+\big):=\displaystyle{\lim_{h\to 0^+} y(t_k+h)}\in E$,
then for every $t\in [t_0, T]$ the following conditions hold:
\begin{itemize}
\item[(B1)] $y_t \in \B$, for every $t\in [t_0,T]$;
\sm

\item[(B2)] the function $t \mapsto y_t$ is continuous on $[t_0,T]\setminus \T$; 
%\red{qui come in Graef; poi, nell'applicazione, verrà ovunque come in GrRy};
\sm

\item[(B3)] $\|y_t\|_\B \leq K(t-t_0)\, \sup_{t_0\leq s\leq t}\|y(s)\| + M(t-t_0)\|y_{t_0}\|_\B$, for every $t\in [t_0,T]$,
\\
where $K,M:[0,+\infty)\to [0,+\infty)$ are independent of $y$, $K$ is strictly positive and continuous, and $M$ is locally bounded;
\sm

\item[(B4)] there exists $H>0$ such that $\|y(t)\|\le H\|y_t\|_\B$.

%\item[(B5)] there exists $m>0$ such that, for every sequence $\{\phi_n\}_{n=1}^{\infty}\subset\B$ with $\|\phi_n-\phi_0\|_\B\rightarrow 0$, the set $\{\phi_n(\theta)\}_{n=1}^{\infty}$ is relatively compact in $E$ for every $\theta\in [-m,0]$. 
\end{itemize}
\end{itemize}
\end{definition} 

\subsection{The generalized phase space for fading memory systems}
\label{ss-Btau}

~

In many real phenomena, including certain biological systems, events occurring further in the past exert a weaker influence on the present state than more recent ones. A standard approach to modeling this feature is to introduce a suitable phase space that constitutes a natural functional setting for the analysis of differential equations with infinite delay of fading type. Such phase spaces have been extensively employed in the literature in the non-impulsive framework; see, for example, the monograph \cite{hmn} and the article \cite{EzSt25}.

In the application we present in this paper (cf. Section \ref{s-A}), we consider the simultaneous presence of an infinite delay of fading type and external forces acting instantaneously and depending on the past evolution of the population itself. These impulsive terms induce abrupt changes in the dynamics of the delay system. Consequently, it is necessary to adopt an appropriate phase space that adequately captures both fading memory effects and impulsive phenomena.

Let $\tau>0$ and $\rho:(-\infty,-\tau)\to \erre$ be a positive Lebesgue integrable function such that 
 there exists a positive locally bounded function $P:(-\infty,0] \to \erre$ such that, for all $\xi\leq 0$,
\begin{equation}\label{e-Prho}
\rho(\xi+\theta) \leq P(\xi) \rho(\theta) \ \mbox{a.e.}\ \theta\in (-\infty,-\tau) .
\end{equation}
A simple example of a function $\rho$ satisfying \eqref{e-Prho} is given by ~$\rho(\theta) := e^{\mu\theta}\, ,\ \mu\in\erre$.

We consider the Lebesgue space
\[
L_\rho((-\infty,-\tau];E):=\left\{
\begin{array}{ll}
\phi:\, (-\infty, -\tau]\to E\ :  
                         & \phi \mbox{  is Lebesgue measurable on $(-\infty,-\tau]$,  }
                         \\
                         & \rho(\cdot)\|\phi(\cdot)\| \in L^1((-\infty,-\tau])
\end{array}\right\},
\]
endowed with the seminorm
$$
\|\phi\|_{L_{\rho}} =  \int_{-\infty}^{-\tau}\rho(\theta)\|\phi(\theta)\|\, d\theta\ .
$$

Further, we consider the normed space of discontinuous functions
\[
{\mathpzc D}([-\tau,0];E) := 
\left\{
\begin{array}{lrl}
\psi:[-\tau, 0]\to E\ :  &  \psi& \mbox{is a piecewise continuous function} 
                    \\ 
                         &  & \mbox{with a finite number of jump discontinuity points}
\end{array}\right\},
\]
%\purple{si può avere un salto senza continuità destra o sinistra, è praticamente lo sp dei passati di GrRy}\\
endowed with the norm
\[
\|\psi\|_{{\mathpzc D}}:=\frac{1}{\tau}\int_{-\tau}^0 \|\psi(\theta)\|\, d\theta.
\]
%is a normed (not Banach) space;
%

We can now define
\[
\B_\tau:=\left\{
%\begin{array}{ll}
\varphi:(-\infty, 0]\to E\ :  \  \varphi\lfloor_{[-\tau,0]} \in {\mathpzc D}([-\tau,0];E) \mbox{ and }  \varphi\lfloor_{(-\infty,-\tau[} \in L_{\rho}((-\infty,-\tau[;E)          
%\end{array}
\right\}.
\]
In $\B_\tau$ we consider the seminorm 
\[
\|\varphi\|_{\B_\tau} =  \left\|\varphi\lfloor_{[-\tau,0]} \right\|_{\mathpzc D}+ \left\| \varphi\lfloor_{(-\infty,-\tau[}  \right\|_{L_\rho}.
\]

It is not difficult to check that the next result holds.

\begin{proposition}
The seminormed space $\left(\B_\tau,\|\cdot\|_{\B_\tau}\right)$ is a generalized phase space.
\end{proposition}

%\begin{proof} 
%Let $y:\,(-\infty, T]\to E$ be a function such
%that $y\lfloor_{[t_0,T]} \in PC_{\mathpzc T}([t_0,T];E)$ and $y_{t_0} \in \B_\tau$.
%
%\begin{itemize}
%\item[(B1)] 
%For any fixed $t\in [t_0,T]$, let us consider the map $\theta\mapsto y(t+\theta)$, $\theta\le 0$. By its properties and by the definition of $\B$, the map $y$ has a finite number of jump discontinuity points on $[t_0-\tau,T]$. Since $$y(t+\cdot)\lfloor_{[-\tau,0]}= y(\cdot)\lfloor_{[t-\tau,t]} \mbox{ and } [t-\tau,t]\subset[t_0-\tau,T],$$
%we get $$y(t+\cdot)\lfloor_{[-\tau,0]}\in {\mathpzc D}([-\tau,0];E).$$
%Moreover, 
 %$y$ is piecewise continuous with jump discontinuities on $[t_0-\tau,T]$, therefore it is Lebesque integrable on the set $[t_0-\tau,T]$ itself. Further, %$\|y(t_0+\cdot)\|$ is measurable on $(-\infty,-\tau]$ and $\rho(\cdot)\|y(t_0+\cdot)\|\in L^1((-\infty,-\tau])$. We can thus deduce that %$y(t+\cdot)\lfloor_{(-\infty,-\tau[}$ belongs to  $L_\rho((-\infty,-\tau];E)$. 
%\\
%Hence, $y_t \in \B$. By the arbitrariness of $t\in [t_0,T]$, the property (B1) follows.\sm
%
%\item[(B2)]
%\item[(B3)] 
%\item[(B4)] 
%\item[(B5)] 
%\end{itemize}
%
%\end{proof}

We will need the following propositions.

\begin{proposition}
If the Banach space $E$ is separable, then the phase space $\B_\tau$ is separable.
\end{proposition}
\begin{proof}
Notice that ${\mathpzc D}([-\tau,0];E)$ is a subspace of $L^1(-\tau,0];E)$, which is a separable space, being $E$ separable. Hence,  ${\mathpzc D}([-\tau,0];E)$ is separable in turn.

On the other hand, we can say the same for $L_\rho((-\infty,-\tau];E)$, which is a subset of $L^1((-\infty,-\tau];E)$.

As a consequence, the product space $L_\rho((-\infty,-\tau];E) \times {\mathpzc D}([-\tau,0];E)$ is separable too.

The thesis follows by observing that the next isomorphism holds
\begin{equation}
\label{e-iso}
(\B_\tau,\|\cdot\|_{\B_\tau}) \cong
(L_\rho((-\infty,-\tau];E),\|\cdot\|_{L_{\rho}})\times ({\mathpzc D}([-\tau,0];E),\|\cdot\|_{{\mathpzc D}})\, .
\end{equation}
\end{proof}

\begin{proposition}
The phase space $\B_\tau$ is complete.
\end{proposition}
\begin{proof}
The spaces $(L_\rho((-\infty,-\tau];E),\|\cdot\|_{L_{\rho}})$ and $({\mathpzc D}([-\tau,0];E),\|\cdot\|_{{\mathpzc D}})$ are complete. Indeed, ${\mathpzc D}([-\tau,0];E)$ and $L_\rho((-\infty,-\tau];E)$ are closed subsets of the complete spaces $L^1((-\tau,0];E)$ and $L^1((-\infty,-\tau];E)$, respectively. Now, since \eqref{e-iso} holds, the space $\B$ is complete.
\end{proof}

\section{Problem Setting and Underlying Assumptions}
\label{s-P}

Let ${\mathpzc T}=\{t_0,\dots,t_m\} \subset [0,T]$ be a set of fixed real numbers with $t_0<t_1<\cdots <t_m$, and $E$ a Banach space. We will use the next spaces:
\begin{itemize}
\item
 the Banach space of piece-wise continuous functions
\[
PC_{\mathpzc T}([t_0,T];E) :=\left\{
\begin{array}{lr}
y:[t_0, T]\to E\ :  &  y_{|\, ]t_{k-1},t_k]} \mbox{ continuous and } \exists \, y\big(t_k^+\big)\in E,\\ 
                    &   k=1\dots, m
\end{array}\right\},
\]
endowed with the uniform norm
\[\|y\|_\infty=\sup_{t\in [t_0,T]}\|y(t)\|;\]

\item a generalized phase space $\left( \B,\|\cdot\|_{\B} \right)$ associated to $\T$; 
\smallskip
\item the solution space 
\[
S((-\infty,T];E):=\left\{
%\begin{array}{ll}
y:(-\infty, T]\to E\ :  \  y\lfloor_{[t_0,T]} \in PC_{\mathpzc T}([t_0,T];E) \mbox{ and }  y_{t_0} \in \B         
%\end{array}
\right\},
\]
endowed with the seminorm 
\[
\|y\|_S =  \left\|y\lfloor_{[t_0,T]} \right\|_\infty + 
\left\| y_{t_0} \right\|_{\B};
\]
\end{itemize}

%SS
\subsection{Assumptions on $(P)$}
%SS

~

We consider the following assumptions on the impulsive Cauchy problem $(P)$.

\begin{itemize}
\item[(A)] $\{A(t)\}_{t\in [t_0,T]}$ is a family of linear not necessarily bounded operators ($A(t):D(A)\subset E\to E$, $t\in [t_0,T]$, $D(A)$ a dense subset of $E$ not depending on $t$)
generating an evolution operator $U:\Delta\to { L}(E)$.
\end{itemize}
On the multimap $F:[t_0,T]\times E \times \B\ri\K v(E)$ we consider the following upper Carath\`eodory type hypotheses:
\begin{itemize}
\item[(F1)] for every $v \in E$, $\varphi\in \B$, the multimap $F(\cdot,v,\varphi):[t_0,T]\ri \K v(E)$ has a strongly measurable selection, i.e. there exists a strongly measurable function $f:[t_0,T]\ri E$
such that $f(t)\in F(t,v,\varphi)$ for a.e. $t\in [t_0,T]$;
 \item[(F2)] for a.e. $t\in[t_0,T]$, the multimap $F(t,\cdot,\cdot):E \times \B \ri \K v(E)$ is u.s.c.;
\end{itemize}
moreover, we require on $F$ also the following assumptions
\begin{itemize}
\item[(F3)] there exists a function $\alpha\in L^1_+([t_0,T])$ such that for a.e. $t \in [t_0,T]$ and every $v\in E$, $\varphi \in \B$ the following sublinearity condition holds
\begin{equation*}
\|F(t,v,\varphi)\|\leq\alpha(t)(1+\|v\|+\|\varphi\|_{\B})\; ,\mbox{ a.e. } t\in[t_0,T];
\end{equation*}
\item[(F4)] there exists a function $\mu\in L^1_+([t_0,T])$ such that, for every $D_1 \subset E$ and $D_2\subset \B$,
\begin{equation*}
\chi(F(t,D_1,D_2))\leq\mu(t)\left(\chi(D_1)+\d\sup_{-\infty\leq \theta\leq 0}
\chi(D_2(\theta))\right)\; ,\mbox{ a.e. } t\in[t_0,T],
\end{equation*}
where $\chi$ is the Hausdorff MNC in $E$.
\end{itemize}

\subsection{Definition of solutions and first results}

We use the following definition of mild solutions for our impulsive functional problem. 

\begin{definition} 
\label{def-mild-sol}
A function $y\in S((-\infty,T];E)$ is a {\em mild solution} for the impulsive Cauchy problem $(P)$ if
\begin{itemize}
\item[(i)] $y(t) = \displaystyle U(t,t_0)y(t_0) +\sum_{t_0<t_k<t}U(t,t_k)I_k(y_{t_k}) +  \int_{t_0}^t U(t,s)f(s)\, ds$,
for every $t\in [t_0,T]$ 
\\
with $f \in L^1([t_0,T];E),\ f(s)\in F(s,y(s),y_s)$ for a.e. $ s\in [t_0,T]$,
\item[(ii)] $y_{t_0}=\varphi^*$,
\item[(iii)] $y(t_k^+)=y(t_k) + I_k(y_{t_k}) \, ,\ k=1,\cdots, m$.
\end{itemize}
\end{definition}
Let $k\in \{1,\dots, m+1\}$. 

\noindent For $q \in C([t_{k-1}, t_k]; E)$, $\xi\in S((-\infty,t_{k-1}];E)$ and $t \in (-\infty,t_k]$ we introduce the following maps
\begin{equation}
\label{qx} q[\xi](t)=\left\{\begin{array}{l}
                        \xi(t)\ ,t\in (-\infty,t_{k-1}[\\
                        q(t)\ ,t\in [t_{k-1},t_k]
                        \end{array}
                        \right.
\end{equation}
\begin{equation}
\label{qxt} q[\xi]_t(\theta)=\left\{\begin{array}{l}
                        \xi(t+\theta)\ ,\theta\in (-\infty,t_{k-1}-t]\\
                        q(t+\theta)\ ,\theta\in ]t_{k-1}-t,0]
                        \end{array}
                        \right.
\end{equation}
Now, we consider the map $j^k : [t_{k-1}, t_k] \times C([t_{k-1}, t_k]; E) \times S((-\infty,t_{k-1}];E) \to \B$ defined by
\begin{equation}
\label{j} 
j^k(t,q,\xi)(\theta)= q[\xi]_t(\theta),\quad \theta \in (-\infty,0].
%
%=\left\{\begin{array}{l}
 %                       \xi(t+\theta)\ ,\theta\in (-\infty,t_{k-1}-t]
  %                      \\
   %                     q(t+\theta)\ ,\theta\in ]t_{k-1}-t,0]
    %                    \end{array}
     %                   \right.
\end{equation}
\begin{proposition}
    \label{j-property}
    For every $k =1,2,\dots,m+1$, the map $j^k : [t_{k-1}, t_k] \times C([t_{k-1}, t_k]; E) \times S((-\infty,t_{k-1}];E)\to \B$ satisfies the following properties
    \begin{itemize}
        \item[(a)] for every $(q,\xi) \in C([t_{k-1}, t_k]; E) \times S((-\infty,t_{k-1}];E)$, $j^k(\cdot,q,\xi)$ is uniformly continuous
        \item[(b)] for every $t \in [t_{k-1},t_k]$, $j^k(t,\cdot,\cdot)$ is Lipschitz continuous.
    \end{itemize} 
\end{proposition}
\begin{proof}
 \noindent Fix $k \in \{1,\dots,m+1\}$.
 
 \noindent Let $q \in C([t_{k-1}, t_k]; E)$ and $\xi \in S((-\infty,t_{k-1}]$. By property $(B2)$, the map $j^k(\cdot,q,\xi):[t_{k-1},t_k] \to \B$ is continuous. Moreover, $\B$ is a seminormed (hence pseudometric) space and so the uniform continuity of $j^k(\cdot, q,\xi)$ follows by Theorem \ref{uniform-cont}.

\noindent Now, let $t \in [t_{k-1},t_k]$, $q_1,q_2 \in C([t_{k-1},t_k];E)$ and $\xi_1,\xi_2 \in S((-\infty,t_{k-1}];E)$. By axiom $(B3)$ we have that 
\begin{equation*}
    \begin{array}{lcl}
\|j^k(t,q_1,\xi_1)-j^k(t,q_2,\xi_2)\|_{\B} & \leq & K(t-t_0)\displaystyle\sup_{t_0 \leq s \leq t}\|q_1[\xi_1](s)-q_2[\xi_2](s)\|+M(t-t_0)\|q_1[\xi_1]_{t_0}-q_2[\xi_2]_{t_0}\|_{\B}\vspace{.1cm}\\
& \leq & K(t-t_0)\displaystyle\sup_{t_0 \leq s \leq t_{k-1}}\|q_1[\xi_1](s)-q_2[\xi_2](s)\|\vspace{.1cm}\\
&&+K(t-t_0)\displaystyle\sup_{t_{k-1} \leq s \leq t}\|q_1[\xi_1](s)-q_2[\xi_2](s)\|\vspace{.1cm}\\
&&+M(t-t_0)\|q_1[\xi_1]_{t_0}-q_2[\xi_2]_{t_0}\|_{\B} \vspace{.1cm}\\
& \leq & K \displaystyle\sup_{t_0 \leq s \leq t_{k-1}}\|\xi_1(s)-\xi_2(s)\|+ K \displaystyle\sup_{t_{k-1} \leq s \leq t}\|q_1(s)-q_2(s)\|\vspace{.1cm}\\
&&+ M\|{\xi_1}_{t_0}-{\xi_2}_{t_0}\|_{\B} \vspace{.1cm}\\
&=& K \|q_1-q_2\|_\infty+\max\{K,M\}\|\xi_1-\xi_2\|_S,
\end{array}
\end{equation*}
where $K=\displaystyle\sup_{0 \leq s \leq T-t_0} K(s)$ and $M=\displaystyle\sup_{0 \leq s \leq T-t_0} M(s)$.
\end{proof}
\noindent Now, we introduce the multivalued superposition operator
$$
P_F^{k,\xi}:C([t_{k-1},t_k];E)\ri\P(L^1([t_{k-1},t_k];E))
$$
defined as
$$
P_F^{k,\xi}(q)=\l\{f\in L^1([t_{k-1},t_k];E)\, : \, f(s)\in
F(s,q(s),j^k(s,q,\xi))\,\mbox{ a.e. } s\in [t_{k-1},t_k]\r\}.
$$
Reasoning as in Theorem 1.3.5 \cite{KOZ} we can prove that the superposition operator $P_F^{k,\xi}$ has non empty values.
\begin{proposition}
\label{existenceselection} 
Under assumptions (F1)-(F3), for every $q \in C([t_{k-1},t_k];E)$ the set $P_F^{k,\xi}(q)$ is non empty.
\end{proposition}
\begin{proof}
Let $k\in \{1,\dots, m+1\}$, $q \in C([t_{k-1}, t_k]; E)$ and $\xi \in S((-\infty,t_{k-1}])$. First of all, we prove that the multivalued map $F(\cdot,q(\cdot),j^k(\cdot,q,\xi)):[t_{k-1},t_k] \to \K v(E)$ admits a strongly measurable selection. 

\noindent By the uniform continuity of $q$ and of $j^k(\cdot,q,\xi)$ (see Proposition \ref{j-property}) there exist two sequences of step functions
\[
q_n : [t_{k-1}, t_k] \to E, \qquad 
r_n : [t_{k-1}, t_k] \to B_\tau
\]
such that
\[
\sup_{t \in [t_{k-1}, t_k]} \| q(t) - q_n(t) \| \xrightarrow[n\to +\infty]{} 0,
\qquad
\sup_{t \in [t_{k-1}, t_k]} \| j^k(t,q,\xi) - r_n(t) \|_\B \xrightarrow[n\to +\infty]{} 0.
\]
By (F1), there exists a sequence of strongly measurable functions 
$f_n : [t_{k-1}, t_k] \to E$ such that
\[
f_n(t) \in F(t, q_n(t), r_n(t)) 
\quad \text{for a.e. } t \in [t_{k-1}, t_k].
\]
Moreover, we consider the sequence  of multifunctions $\{ P_m \}_{m=1}^\infty$, $P_m:[t_{k-1},t_k]\to \K(E^\prime)$, 
\[
P_m(t) := \overline{\bigcup_{n=m}^\infty  f_n(t)},
\]
where $E^\prime=\overline{\mbox sp} \bigcup_{n=1}^\infty f_n([t_{k-1},t_k])\subset E$. 
Note that, since $F(t,\cdot,\cdot)$ is upper semicontinuous with compact values, it maps compact sets into compact sets; by the convergence above and condition (F2),  the sets
\(P_m(t)\) are compact, for every $ t \in [t_{k-1},t_k]$.
Further, the multifunctions $P_m$ are clearly measurable.

\noindent  As $\{ P_m(t) \}_{m=1}^\infty$ is a decreasing sequence of sets and the intersection of measurable multifunctions is measurable, we can define $P:[t_{k-1},t_k] \to \K(E^\prime)$, as
\[
P(t) = \bigcap_{m=1}^\infty P_m(t), \quad t \in [t_{k-1}, t_k],
\]
which is a measurable multifunction taking compact values in $E^\prime$. Moreover, by the upper semicontinuity of $F(t,\cdot,\cdot)$
$$
\begin{array}{lcl}
P(t) & = & \displaystyle\bigcap_{m=1}^\infty P_m(t) = \bigcap_{m=1}^\infty \overline{\bigcup_{n=m}^\infty  f_n(t)} \\
& \subset & \displaystyle\bigcap_{m=1}^\infty \overline{\bigcup_{n=m}^\infty F(t,q_n(t),r_n(t))} \\
& \subset & F(t,q(t),j^k(t,q,\xi)),\quad t \in [t_{k-1}, t_k]
\end{array}
$$
Since $E^\prime$ is separable, there exists a strongly measurable selection 
$p(t) \in P(t)$ such that
\[
p(t) \in F(t, q(t), j^k(t, q,\xi)) \quad \text{for a.e. } t \in [t_{k-1}, t_k].
\]
This provides the desired strongly measurable selection. Moreover, by $(F3)$ it follows that
$$
\|p(t)\| \leq \|F(t,q(t),j^k(t,q,\xi))\| \leq \alpha(t)(1+\|q(t)\|+\|j^k(t,q,\xi)\|_\B)
$$
with $\alpha \in L^1_+([t_{k-1},t_k])$ and $q(\cdot)$, $j^k(\cdot,q,\xi)$ continuous functions, implying $p \in L^1([t_{k-1},t_k];E)$. Thus $p \in P^{k,\xi}_F(q)$.
\end{proof}
Moreover, reasoning as in \cite{KOZ} Lemma 5.1.1 it is possible to prove the following regularity result for the superposition multioperator.
\begin{proposition}
\label{closeness} 
Under assumptions (F1)-(F3), for every $k=1,\dots,m+1$ the superposition operator $P_F^{k,\xi}$ is strongly-weakly closed .
\end{proposition}
\begin{proof}
 Let $k \in \{1,\dots,m+1\}$ and $\xi\in S((-\infty,t_{k-1}];E)$ be fixed. Let $\{f_n\}_{n=1}^{\infty} \subset L^1([t_{k-1},t_k];E)$ $f_n \rightharpoonup f_0$ and $\{q_n\}_{n=1}^{\infty} \subset C([t_{k-1},t_k];E)$ $q_n \to q_0$ we will prove that $f_0 \in P_F^{k,\xi}(q_0)$.

 \noindent By Mazur's Lemma the weak convergence of $\{f_n\}_{n=1}^{\infty}$ to $f_0$ implies the existence of a double sequence of nonnegative numbers ${\{\alpha_{ij}\}_{i=1}^\infty}_{j=1}^\infty$ such that 
 \begin{itemize}
     \item[-] $\displaystyle\sum_{j=1}^\infty \alpha_{ij}=1$ for all $i=1,2,\dots$;
     \item[-] there exists a number $j_0(i)$ such that $\alpha_{ij}=0$ for all $j \geq j_0(i)$;
     \item[-] the sequence $\{\widetilde{f}_i\}_{i=1}^{\infty}$ defined as $\widetilde{f}(t)=\displaystyle\sum_{j=1}^\infty \alpha_{ij}f_j(t)$ converges strongly to $f_0$ in $L^1[t_{k-1},t_k];E)$.
\end{itemize}
Up to subsequence we can assume that $\{\widetilde{f}_i\}_{i=1}^{\infty}$ converges almost everywhere to $f_0$ on $[t_{k-1},t_k]$.
By assumption $(F2)$ and Proposition \ref{j-property}, for a.e. $t \in [t_{k-1},t_k]$, the multivalued map $F(t,\cdot,j^k(t,\cdot))$ is an upper semicontinuous map, being the composition of an uppersemicontinuous map and a Lipshitz map. Thus, for a given $\varepsilon > 0$ there exists an integer $i_0=i_0(\varepsilon,t)$ such that
$$
F(t,q_i(t),j^{k}(t,q_i,\xi)) \subset W_\varepsilon(F(t,q_0(t),j^k(t,q_0,\xi))),
$$
for $i \geq i_0$, where $W_\varepsilon$ denotes the $\varepsilon$-neighborhood of a set. Then 
$$
f_i(t) \in F(t,q_i(t),j^{k}(t,q_i,\xi)) \subset W_\varepsilon(F(t,q_0(t),j^k(t,q_0,\xi)))
$$
for $i \geq i_0$ and hence, being $W_\varepsilon$ a convex neighborhood, also
$$
\widetilde{f}_i(t) \in W_\varepsilon(F(t,q_0(t),j^k(t,q_0,\xi)))
$$
for $i \geq i_0$. Thus, by the almost everywhere convergence of $\{\widetilde{f}_i\}_{i=1}^{\infty}$, it follows
$$
f_0(t) \in F(t,q_0(t),j^k(t,q_0,\xi)),
$$
i.e. $f_0 \in P^{k,\xi}_F(q_0)$.
\end{proof}
\noindent Observing that by Proposition \ref{j-property}, for every $k=1,\dots,m+1$ and every $t \in [t_{k-1},t_k]$ the map $j^k(t,\cdot,\cdot):C([t_{k-1},t_k];E) \times S((-\infty,t_{k-1}];E) \to \B$ is Lipshitz continuous, reasoning as in Proposition \ref{closeness} it is possible to prove also the following result regarding the closure of the superposition multioperator.
\begin{proposition}
    \label{closeness-2}
    Under assumptions $(F1)-(F3)$, let $k\in \{1,\dots,m+1\}$, $\{\xi_n\}_{n=1}^\infty \subset S((-\infty,t_{k-1}])$, $\xi_n \to \xi$, $\{q_n\}_{n=1}^\infty \in C([t_{k-1},t_k];E)$, $q_n \to q$, $\{f_n\}_{n=1}^\infty \subset L^1([t_{k-1},t_k];E)$, $f_n \in P^{k,\xi_n}_F(q_n)$ for every $n \in \mathbb{N}$, $f_n \rightharpoonup f$, then $f \in P^{k,\xi}(q)$. 
\end{proposition}

\section{Main results}
\label{main}

\subsection{Existence results}
\begin{theorem}
\label{existence} Under assumptions (A) and (F1)-(F4) the problem $(P)$ has at least one mild solution on $(-\infty,T]$.
\end{theorem}

\begin{proof}
We prove the existence of a mild solution of problem (P) splitting the problem in the intervals $[t_{k-1},t_k]$, $k=1,\dots,m+1$. First of all we prove the existence on the interval $[t_0,t_1]$. 

\noindent We define the solution operator $\Gamma_1:C([t_0,t_1];E) \to \P(C([t_0,t_1];E))$ defined as
$$
\Gamma_1(q)(t)=\left\{U(t,t_0) \varphi^0(t_0)+\displaystyle\int_{t_0}^t U(t,s) f(s) \, ds,\; f \in P^{1,\varphi^0}_F(q) \right\},
$$
where we put $\varphi^0=\varphi^*$.

\noindent {\bf Step 1.} $\Gamma_1$ is a closed multivalued operator.

\noindent Let $\{q_n\}_{n=1}^{\infty} \subset C([t_0,t_1];E)$ and $\{y_n\}_{n=1}^{\infty} \subset C([t_0,t_1];E)$ be such that $q_n \to q_0$, $y_n \in \Gamma^1(q_n)$ and $y_n \to y_0$. 

\noindent By the fact that $y_n \in \Gamma_1(q_n)$ we have that there exists $f_n \in P^{1,\varphi^0}_F(q_n)$ such that
$$
y_n(t)=U(t,t_0)\varphi^0(t_0)+\displaystyle\int_{t_0}^t U(t,s) f_n(s)\,ds.
$$
By assumption $(F4)$ we have that
$$
\begin{array}{lcl}
\chi(f_n(t)) & \leq & \mu(t) \left(\chi(\{q_n(t)\}_{n=1}^{\infty})+\sup_{-\infty < \theta \leq 0} \chi\left(\{j^1(t,q_n,\varphi^0)(\theta)\}_{n=1}^{\infty}\right)\right) \\
& = & \mu(t) \left(\chi(\{q_n(t)\}_{n=1}^{\infty})+\sup_{t_{0}-t < \theta \leq 0} \chi\left(\{q_n(t+\theta)\}_{n=1}^{\infty}\right)\right)
\end{array}
$$
By the convergence of $\{q_n\}_{n=1}^{\infty}$ in $C([t_0,t_1];E)$ we have that $\{f_n(t)\}_{n=1}^{\infty}$ is relatively compact and by assumption $(F3)$ and property $(B3)$ we have that
$$
\begin{array}{lcl}
\|f_n(t)\| & \leq & \alpha(t)(1+\|q_n(t)\|+\|j^1(t,q,\varphi^0)\|_{\B}) \\
& \leq & \alpha(t) (1+\|q_n(t)\| +K(t-t_0)\sup_{t_0 \leq s \leq t}\|q_n(t)\|+M(t-t_0)\|\varphi^0_{t_0}\|_{\B}).
\end{array}
$$
Again, by the convergence of $\{q_n\}_{n=1}^{\infty}$ in $C([t_0,t_1];E)$ we have that the sequence $\{q_n\}_{n=1}^{\infty}$ is bounded and so we have that the sequence $\{f_n\}_{n=1}^{\infty}$ is integrably bounded, thus it is a semicompact sequence. So, by Theorem \ref{wcomp} we have that there exists a map $f \in L^1([t_0,t_1];E)$ such that $\{f_n\}_{n=1}^{\infty}$ weakly converges to $f$ in $L^1([t_0,t_1];E)$ and by Theorem \ref{p-t} we have that
$$
\displaystyle\int_{t_0}^t U(t,s) f_n(s)\, ds \to \displaystyle\int_{t_0}^t U(t,s) f(s)\, ds.
$$
Moreover, $F(t,\cdot,j^1(t,\cdot,\varphi^0))$ is an upper semicontinuous map, being the composition of an uppersemicontinuous map (see assumption (F2)) and a Lipshitz map (see Proposition \ref{j-property}). So, by Proposition \ref{closeness} we have that $f\in P^{1,\varphi^0}_F(q_0)$. In conclusion, by the uniqueness of the limit we get that
$$
y_0(t)=U(t,t_0)q_0(t_0)+\displaystyle\int_{t_0}^t U(t,s)f(s)\,ds,
$$
i.e. the claimed result.

\vspace{.2cm}
\noindent {\bf Step 2.} $\Gamma_1$ is a condensing multivalued operator.

\noindent Let $\Omega \subset C([t_0,t_1];E)$ be a bounded set such that
\begin{equation}
   \label{ineq}
\nu_1(\Gamma_1(\Omega)) \geq \nu_1(\Omega), 
\end{equation}
with
$$
\nu_1(\Omega)=\displaystyle\max_{D \in \mathcal{D}(\Omega)}\left(\gamma_1(D(t)),\mbox{mod}_C(D)\right),
$$
where $\mathcal{D}(\Omega)$ is the collection of all denumerable subsets of $\Omega$, $L^1 > 0$ is chosen so that
$$
q:= 2 D \displaystyle\sup_{t \in [t_0,t_1]} \int_{t_0}^t e^{-L_1(t-s)} \mu(s)\,ds < 1,
$$
$\gamma_1(D(t))$ is the measure of non compactness defined as
$$
\gamma_1(D(t))=\sup_{t \in [t_0,t_1]} e^{-L^1(s-t_0)}\chi(D(t))
$$
and $\mbox{mod}_C(D)$ is the modulus of equicontinuity defined as
$$
\mbox{mod}_C(D)=\displaystyle\lim_{\delta \to 0} \; \sup_{q \in D} \; \max_{|t'-t''|\leq \delta} \|q(t')-q(t'')\|.
$$
Let the maximum on the left hand side of the inequality \eqref{ineq} be achieved for the countable set $D=\{y_n\}_{n=1}^\infty,$ then there exist a sequence $\{q_n\}_{n=1}^\infty$ and a sequence $\{f_n\}_{n=1}^\infty$ with $f_n \in P^{1,\varphi^0}_F(q_n)$ such that $y_n \in \Gamma_1(q_n)$. For every $s \in [t_0,t_1]$, we have that
$$
\begin{array}{lcl}
\chi(\{f_n(s)\}_{n=1}^\infty) & \leq & \mu(s)\left(\chi(\{q_n(s)\}_{n=1}^\infty)+\displaystyle\sup_{-\infty < \theta \leq 0} \chi\left(\{j^1(s,q_n,\varphi^0)(\theta)\}_{n=1}^\infty\right)\right) \\
& = & e^{L_1(s-t_0)} \mu(s)\left(e^{-L_1(s-t_0)}\chi(\{q_n(s)\}_{n=1}^\infty)+e^{-L_1(s-t_0)}\displaystyle\sup_{t_0-s \leq \theta \leq 0} \chi\left(\{q_n(s+\theta)\}_{n=1}^\infty\right)\right) \\
& \leq & e^{L_1(s-t_0)} \mu(s)\left(e^{-L_1(s-t_0)}\chi(\{q_n(s)\}_{n=1}^\infty)+e^{-L_1(s-t_0)}\displaystyle\sup_{t_0 \leq \tau \leq t_1} \chi\left(\{q_n(\tau)\}_{n=1}^\infty\right)\right) \\
& \leq & e^{L_1(s-t_0)} \mu(s) 2 \gamma_1(\{q_n\}_{n=1}^\infty).
\end{array}
$$
Now, applying Lemma 3 in \cite{CR-2005} we have that
$$
\displaystyle\chi\left(\int_{t_0}^t U(t,s)f_n(s)\, ds\right) \leq 2D \gamma(\{q_n\}_{n=1}^\infty \int_{t_0}^t e^{L(s-t_0)} \mu(s)\,ds.
$$
Thus, from \eqref{ineq} we have that
$$
\gamma_1(\{q_n\}_{n=1}^\infty) \leq 2 D \gamma_1(\{q_n\}_{n=1}^\infty) \displaystyle\sup_{t \in [t_0,t_1]} \int_{t_0}^t e^{-L_1(t-s)} \mu(s)\,ds = q \gamma_1(\{q_n\}_{n=1}^\infty).
$$
Being $q < 1$ we have that
$$
\gamma_1(\{q_n\}_{n=1}^\infty) = 0.
$$
So, by (F3) and (F4) we have that $\{f_n\}_{n=1}^\infty$ is a semicompact sequence, thus by Proposition \ref{p-t} we have that $\{y_n\}_{n=1}^\infty$ is relatively compact, therefore
$$
\gamma_1(\{y_n\}_{n=1}^\infty) =0 \quad \mbox{and} \quad \mbox{mod}_C(y_n)=0,
$$
i.e $\nu_1(\Gamma_1(\Omega))=(0,0)$, implying, by \eqref{ineq}, $\nu_1(\Omega)=(0,0)$, i.e. $\Omega$ is a relatively compact set.

\vspace{.2cm}
{\bf Step 3.} Now we prove that there exists a constant $r>0$ such that the multioperator $\Gamma_1$ maps the ball $B_r(0)$ into itself.

\noindent We introduce the equivalent norm in the space $C([t_0,t_1];E)$
$$
\|q\|_*=\max_{t \in [t_0,t_1]} e^{-N_1 t} \|q(t)\|,
$$
where $N_1$ is chosen so that
$$
\ell_1:=\max_{t_0 \leq t \leq t_1} D \displaystyle\int e^{-N_1(t-s)}\alpha(s)(1+K_1)\,ds < 1
$$
with $K_1 > 0$ to be determined later. So, we consider the ball $B_r(0)=\{q \in C([t_0,t_1];E)\, : \, \|q\|_* \leq r\}$.

\noindent Let $q \in B_r(0)$ and $y \in \Gamma_1(q)$, we have that there exists $f \in P^{1,\varphi^0}_F(q)$ such that
$$
\begin{array}{lcl}
\|y(t)\| & \leq & \|U(t,t_0)\| \|\varphi^0(t_0)\| + \displaystyle\int_0^t \|U(t,s)\| \|f(s)\|\,ds \\
& \leq & D \|\varphi^0(t_0)\| + \displaystyle\int_{t_0}^t \|U(t,s)\| \alpha(s)(1+\|q(s)\|+\|j^1(s,q,\varphi^0)\|) \,ds \\
& \leq & D \|\varphi^0(t_0)\| + D \displaystyle\int_{t_0}^t \alpha(s) (1+\|q(s)\| +K(s-t_0)\sup_{t_0 \leq \tau \leq s}\|q(\tau)\|+M(s-t_0)\|\varphi^0_{t_0}\|_{\B})\,ds,
\end{array}
$$
where the last inequality is due to property $(B3)$ of the space $\B$.
Hence, it follows that
$$
\begin{array}{lcl}
\|y(t)\| & \leq & D \|\varphi^0(t_0)\| + D \|\alpha\|_{L^1_+([t_0,t_1])}+ \displaystyle D \int_{t_0}^t \alpha(s) (\|q(s)\|+ K_1 \sup_{t_0 \leq \tau \leq s}\|q(\tau)\|)\,ds\\
&&+M_1 \|\alpha\|_{L^1_+([t_0,t_1])} \|\varphi^0_{t_0}\|_{\B},
\end{array}
$$
where $K_1=\displaystyle\max_{0 \leq \tau \leq t_1-t_0} K(\tau)$ and $M_1=\displaystyle\sup_{0 \leq \tau \leq t_1-t_0} M(\tau)$, such bounds do exist by the continuity of $K$ and the local boundedness of $M$.
Now, denoting by $C_1=D \|\varphi^0(t_0)\| + D \|\alpha\|_{L^1_+([t_0,t_1])}+M_1 \|\alpha\|_{L^1_+([t_0,t_1])} \|\varphi^0_{t_0}\|_{\B})$, it follows that
$$
e^{-N_1 t} \|y(t)\| \leq C_1 e^{-N_1 t}+D e^{-N_1 t} \displaystyle D \int_{t_0}^t \alpha(s) (1+ K_1) \sup_{t_0 \leq \tau \leq s}e^{N_1 \tau} e^{-N_1 \tau}\|q(\tau)\|)\,ds,
$$
implying that
$$
e^{-N_1 t} \|y(t)\| \leq C_1 + D \|q\|_* \displaystyle\int_{t_0}^t e^{-N_1(t-s)} \alpha(s)(1+K_1)\,ds \leq C_1 + r \ell_1 \leq r,
$$
where the last inequality is due to the choice of $N_1$ and $r>0$ is chosen so that
$$
r (1-\ell_1) \geq C_1,
$$
obtaining that $\|y\|_* \leq r$.

\vspace{.2cm}
\noindent {\bf Step 4.} The operator $\Gamma_1$ admits a fixed point.

\noindent Applying Theorem \ref{fixedpoint} we get the existence of a fixed point $y_1 \in \Gamma_1(y_1)$, i.e. the existence of a selection $f_1 \in P^{1,\varphi^0}_F(y_1)$ such that
$$
y_1(t)= U(t,t_0) \varphi^0(t_0)+\displaystyle\int_{t_0}^t U(t,s) f_1(s) \, ds.
$$
Thus, the map $y^1[\varphi^0]$ is a mild solution of the problem $(P)$ on $(-\infty,t_1]$.

\vspace{.2cm}
\noindent {\bf Step 5.} Now we proceed considering the interval $[t_1,t_2]$ and the operator 
$\Gamma_2:C([t_1,t_2];E) \to \P(C([t_1,t_2];E))$ defined as
$$
\Gamma_2(q)(t)=\left\{U(t,t_1) \varphi^1(t_1)+I_1(y^1_{t_1})+\displaystyle\int_{t_1}^t U(t,s) f(s) \, ds,\; f \in P^{2,\varphi^1}_F(q) \right\},
$$
with
$$
\varphi^1(t)=\left\{\begin{array}{ll}
                 \varphi^0(t) & t \in (-\infty,t_0] \\
                 y_1(t) & t \in ]t_0,t_1].
                 \end{array}
                 \right.
$$
Reasoning as above we can prove the existence of a fixed point $y_2 \in \Gamma_2(y_2)$, i.e. the existence of a selection $f_1 \in P^{1,\varphi^0}_F(y_1)$ such that
$$
y_2(t)= U(t,t_1) \varphi^1(t_1)+I_1(y^1_{t_1})+\displaystyle\int_{t_1}^t U(t,s) f_2(s) \, ds.
$$
Thus, the map $y^2[\varphi^1]$ is a mild solution of the problem $(P)$ on $(-\infty,t_2]$.

\vspace{.2cm}
{\bf Step 6.} Repeating this procedure we get a fixed point $y_k \in \Gamma_k(y_k)$ for every $k =1,\dots,m+1$. Gluing together all these fixed points we get a solution of problem (P).
\end{proof}

\subsection{Compactness of the solution set}

~

If we assume that the impulse functions are continuous, we have that the solution set of problem $(P)$ is compact. Namely, we prove the following result.
\begin{theorem}
\label{t:comp} Under assumptions (A), (F1)-(F4) and if the impulse functions $I_k:\B\ri E, \,
k=1,\dots,m$ are continuous, then the solutions set is a compact subset of the space $S((-\infty,T];E)$.
\end{theorem}
\begin{proof}
By Theorem \ref{existence} we have proven that the set of solution of problem $(P)$ is non-empty. We can characterize it by
\begin{equation}
\label{sigma}
\Sigma=\bigcup\left\{\Sigma^{m+1}_{\varphi^{m}} \, :y^1\in\Sigma^1_{\varphi^0};y^2\in\Sigma^2_{\varphi^1}\cdots;y^{m}\in\Sigma^{m}_{\varphi^{m-1}}\right\}
\end{equation}
where for every $k=1,\dots,m+1$
$$
\varphi^k(t)=\left\{\begin{array}{ll}
                \varphi^{k-1} & t \in (-\infty,t_{k-1}] \\
                y^k(t) & t \in ]t_{k-1},t_k]
                \end{array}
                \right.
$$
and $\Sigma^k_{\varphi^k}$ is the solution set of the problem $(P)$ considered on $(-\infty,t_k]$, i.e.
$$
\Sigma^k_{\varphi^{k-1}}=\{y^k[\varphi^{k-1}], y^{k-1} \in \mbox{Fix}(\Gamma_{k-1})\}.
$$.

\noindent First of all, we define the multifunction $H^1:\Sigma^1_{\varphi^0}\to \P(C([-\tau,t_2];E))$ as
$$
H^1(\varphi^1)=\Sigma^2_{\varphi^1}\ .
$$
Since in Step 3 of the proof of Theorem \ref{existence} we have proven that there exists $r > 0$ such that
$$
\mbox{Fix}(\Gamma^1) \subset B_r(0),
$$
from Theorem \ref{fixedpoint}, we have that the set $\Sigma^1_{\varphi^0}$ is compact. Moreover, analogously, it is possible to prove that for every $\varphi^1 \in \Sigma^1_{\varphi^0}$, $H^1(\varphi^1)$ is compact. Now, we prove that $H^1$ is upper semicontinuous. 

\noindent Note that defining the multifunction
$Q^1:\Sigma^1_{\varphi^0}\rightarrow \K(C([t_1,t_2];E))$ as
$$
Q^1(\varphi^1)={\Sigma^2_{\varphi^1}}\lfloor_{[t_1,t_2]}\ ,
$$
the multifunction $H^1$ can be written as the composition of the
multimap $P^1:\Sigma^1_{\varphi^0}\rightarrow \K(\Sigma^1_{\varphi^0}\times C([t_1,t_2];E))$
$$
P^1(\varphi^1)=\{\varphi^1\}\times Q^1(\varphi^1)
$$
with the continuous map $\eta^1:P^1(\Sigma^1_{\varphi^0})\rightarrow S((-\infty,t_2];E)$
$$
\eta^1(\varphi^1,q)=q[\varphi^1]
$$
(see (\ref{qx})). So, we first prove that the multifunction $Q^1$ is upper semicontinuous.

\noindent We assume by contradiction that there exists $\overline{y}^1\in {\Sigma^2_{\varphi^1}}_{|\, [t_1,t_2]}$ such that $Q^1$ is not u.s.c. in $\overline{\varphi}^1$. Therefore there exist $\overline{\varepsilon}>0$ and two sequences $\{\varphi_n^1\}_{n=1}^\infty$, $\varphi_n^1\rightarrow \overline{\varphi}^1$ in $S((-\infty,t_1];E)$, and
$\{y_n^2\}_{n=1}^\infty,\ y_n^2\in {\Sigma^2_{\varphi^1_n}}\lfloor_{[t_1,t_2]}$, such that
\begin{equation}
\label{notusc} 
y_n^2\notin B\left({\Sigma^2_{\overline{\varphi}^1}}\lfloor_{[t_1,t_2]},\overline{\varepsilon}\right)\ ,\ n\geq 1
\end{equation}
Since $\{y_n^2\}_{n=1}^\infty$ is a sequence of solutions, we
have:
\begin{equation}
\label{solution2} 
y_n^2(t) = U(t,t_1)[\varphi^1_n(t_1)+I_1({\varphi_n^1}_{t_1})] + \int_{t_1}^t U(t,s)f_n^2(s)\, ds  \quad ,\ t\in [t_1,t_2]
\end{equation}
where $f_n^2 \in P^{2,\varphi^1_n}_F(y_n^2)$. \np Then, for $s\in [t_1,t]$, $(F4)$ yields
$$
\begin{array}{lcl}
\chi(f_n^2(t)) & \leq & \mu(t) \left(\chi(\{y_n^2(t)\}_{n=1}^{\infty})+\sup_{-\infty < \theta \leq 0} \chi\left(\{j^2(t,y_n^2,\varphi_n^1)(\theta)\}_{n=1}^{\infty}\right)\right) \\
& = & \mu(s) \left(\displaystyle\sup_{-\infty\leq \theta \leq t_1}
\chi\left(\{\varphi_n^1(\theta)\}_{n=1}^\infty\right)+
\displaystyle\sup_{t_1\leq \theta \leq s}
\chi\left(\{y_n^2(\theta)\}_{n=1}^\infty \right)\right)
\end{array}
$$
The MNC $\chi\left(\{\varphi_n^1(\theta)\}_{n=1}^\infty \right)=0$, since $\{\varphi_n^1\}_{n=1}^\infty$ is a converging sequence, then
\begin{equation*}
\chi\left(\{f_n^2(s)\}_{n=1}^\infty\right)\leq e^{Ls}\mu(s)
\displaystyle\sup_{t_1\leq \theta \leq t_2} e^{-L\eta}
\chi\left(\{y_n^2(\theta)\}_{n=1}^\infty\right)\ .
\end{equation*}
for a suitable $L > 0$. Now using similar argument as in Step 2 of the proof of Theorem \ref{existence}, it is possible to prove that the set $\{y_n^2\}_{n=1}^\infty$ is relatively compact in $C([t_1,t_2];E)$. Therefore w.l.o.g. we can assume that there exists $\overline{y}^2\in C([t_1,t_2];E)$
such that $y^2_n\rightarrow \overline{y}^2$ in $C([t_1,t_2];E)$. Now we prove that $\overline{y}^2\in Q^1(\overline{\varphi}^1)$.

\noindent As in Step 1 of the proof of Theorem \ref{existence} it is possible to prove that there exists
$\overline {f}^2\in L^1([t_1,t_2];E)$ such that $f_n^2\rightharpoonup \overline{f}^2\in L^1([t_1,t_2];E)$. Now, by Proposition \ref{closeness-2} we have
\begin{equation*}
\overline{f}^2\in P^{2,\varphi^1}_F(\overline{y}^2).
\end{equation*}
From Proposition \ref{p-t} and the fact that the function $I_1$ is continuous, by considering the limit in both sides of
(\ref{solution2}) we get
\begin{equation*}
\overline{y}^2(t) = U(t,t_1)[\varphi^1(t_1)+I_1(\varphi^1_{t_1})] + \int_{t_1}^t U(t,s)\overline{f}^2(s)\, ds  \quad ,\ t\in [t_1,t_2],
\end{equation*}
that is $\overline{y}^2\in {\Sigma^2_{\overline{\varphi}^1}}\lfloor_{[t_1,t_2]}=Q^1(\overline{\varphi}^1).$
\noindent The fact that $y_n^2\rightarrow \overline{y}^2\in Q^1(\overline{\varphi}^1)$ leads a contradiction with (\ref{notusc}). Therefore, by applying well known results on composition and cartesian product of multimaps (see e.g. \cite{KOZ}, Theorem 1.2.12 and Theorem 1.2.8), we can conclude that $H^1$ is
u.s.c.. Hence, the set $H_1(\Sigma^1_{\varphi^0})$ is compact, being the image of an upper semicontinuous multimap of the compact set $\Sigma^1_{\varphi^0}$. In conclusion the set
$$
\bigcup_{\varphi^1 \in \Sigma^1_{\varphi^0}} \Sigma^2_{\varphi^1}=H^1(\Sigma^1_{\varphi^0})
$$
is compact.

\noindent By iterating this process we obtain the compactness of the
solution set $\Sigma$ (cf. (\ref{sigma})) on the whole interval $(-\infty,T]$.
\end{proof}

As a consequence of the compactness result Theorem \ref{t:comp}, we can establish the existence of optimal solutions to problem $(P)$. In other words, we are establishing the minimization or maximization of a cost functional associated with problem $(P)$.

\begin{theorem}
\label{t:opt}
Assume the same hypotheses as Theorem \ref{t:comp} and let ${\mathcal J}:PC_\T ([t_0,T];E)\to \erre$ be a cost functional for $(P)$.
\\
If ${\mathcal J}$ is lower semicontinuous, then there exists a mild solution $y_*$ of $(P)$ such that 
\[{\mathcal J}(y_*)=\min_{y\in{\mathcal S}} {\mathcal J}(y);\]
if ${\mathcal J}$ is upper semicontinuous, then there exists a mild solution $y^*$ of $(P)$ such that 
\[{\mathcal J}(y^*)=\max_{y\in{\mathcal S}} {\mathcal J}(y),\]
where ${\mathcal S}$ is the set of all the mild solutions of $(P)$.
\end{theorem}

\section{Application}
\label{s-A}

In this section, we provide an example of an application of our result on optimal trajectories for an energy functional associated with an impulsive problem with infinite delay. The example is taken from population dynamics theory and has already been studied in other settings by some authors (see, for instance, \cite{EzSt25, Ru26}); however, it can serve as a prototype for many other possible applications, such as heat diffusion equations or flexible robotic arm models. More generally, it applies to all situations in which the differential equation describing the phenomenon under consideration can be written as the sum of a linear part (e.g., the Laplacian or the biharmonic operator) and a nonlinear part depending on the past history of the solution trajectory.

\subsection{Optimal solutions to the populations dynamics model with impulses and infinite delay}

~

We deal with a population evolution model in which the present state of the system depends on its entire history and allow for the presence of external forces acting instantaneously and depending on the past evolution of the population itself. 

We seek the existence (in the mild sense) of optimal solutions for an energy functional associated to the impulsive initial value problem with feedback controls 
\begin{equation}
\label{p-appl}
\begin{cases}
\frac{\partial }{\partial t}u(t,x)= -b(t,x)u(t,x) + g\left(t,u(t,x),\int_{-\infty}^0 u(t+\theta,x)d\theta\right)+ \omega(t,x), 
\\ \hfill \ t \in [t_0,T]\setminus\{t_1,\dots, t_m\}, \mbox{ a.e. } x\in  [0,1],\\ 
\omega(t,\cdot)\in \Omega(u(t,\cdot)),\  t \in [t_0,T],\\
\\
u(t,x)=\psi^*(t,x),\, t \in (-\infty,t_0], \mbox{ a.e. } x\in [0,1],\\
\\
u(t_k^+,x)=u(t_k,x) + \mathcal{I}_k \left( \int_{-\infty}^0 u(t_k+\theta,x)d\theta \right) ,\mbox{ a.e. } x\in [0,1], \, k=1,\dots, m,
\end{cases}
\end{equation} 
where: $u(t,x)$ the population density at time $t$ and place $x$; $b:[t_0,T]\times [0,1]\to \erre^+$ is the removal coefficient (involving death rate and migration); $g:[t_0,T]\times \erre\times \erre\to \erre$ represents the nonlinear population development law; $\{t_0,\dots,t_m\}$ is a given subset of $[0,T]$ with $t_0<t_1<\dots<t_m$;
%$\mu\in\erre$ is the fading memory regulator,  
$\Omega:L^2([0,1])\to \P(L^2([0,1]))$ is the multivalued mapping that yields the feedback controls $\omega$;
the function $\psi^*:(-\infty,t_0]\times [0,1]\to \erre$ is the initial datum.

Let ${\mathpzc T}=\{t_0,\dots,t_m\}$. Consider the separable Banach space 
\begin{equation}
\label{e-E}
    E:=L^2([0,1])
\end{equation}
with its usual norm $\|v\|_E^2=\int_0^1 v^2(t)\, dt$, and the corresponding generalized phase space $\B_{\tau}$ (see Section \ref{ss-Btau}).

We assume that the trajectories $u$ are such that $u(t,\cdot)\in E$ for every $t\le T$ and that the initial datum $\psi$ is such that $t\mapsto \psi^*(t,\cdot)$ belongs to $\B_{\tau}$.

\begin{proposition}(cf. \cite[Proposition 4.1]{R22})\label{p:A}
Assume that the function $b:[t_0,T] \times [0,1]\to \erre$ satisfies properties 
\begin{itemize}
\item[(b1)] $b$ is measurable;
\item[(b2)] there exists $s\in L^1([t_0,T])$ such that
$$
0<b(t,x)\le s(t)\, ,\ \mbox{for every } t\in [t_0,T], \mbox{ a.e. } x\in [0,1];
$$
\item[(b3)] for every $x\in [0,1]$, the function $b(\cdot,x):[t_0,T]\to\erre^+$ is continuous.
\end{itemize}
Then, the family $\{A(t)\}_{t\in [t_0,T]}$, $A(t):E\to E$, $t \in [t_0,T]$, defined by 
\begin{equation}\label{e:Ab}
A(t)v(x)=-b(t,x)v(x),\ v\in E,\, x\in [0,1],
\end{equation} 
satisfies property (A).
\\
Moreover, the (noncompact) evolution system generated by  $\{A(t)\}_{t\in [t_0,T]}$ is given by
\begin{equation*}
\label{e:Uv}
[U(t,s)v](x)= e^{\int_s^t-b(\sigma,x)d\sigma}v(x),
\end{equation*}
for every $ t_0\le s\le t\le T , \, v\in E,\,  x\in [0,1]$.
\end{proposition}

\begin{definition}
\label{d-stc}
We say that a mild solution of \eqref{p-appl} is a  pair $(u,\omega)$, mild trajectory  $u:(-\infty,T] \times [0,1]\to \erre$  and control $\omega:[t_0,T]\times [0,1]\to \erre$, such that 
\begin{align}
i) \ &u(t,x)=e^{\int_{t_0}^t-b(\sigma,x)d\sigma} \psi^*(t_0,x) + \displaystyle \sum_{t_0 < t_j < t} e^{\int_{t_j}^t-b(\sigma,x)d\sigma} {\mathcal I}_j\left( \int_{-\infty}^0 u(t_j+\theta,x)d\theta \right) +\\
  & \hspace{1,5cm}  + \int_{t_0}^t e^{\int_s^t-b(\sigma,x)d\sigma} g\left(s,u(s,x),\int_{-\infty}^0 u(s+\theta,x)\, d\theta\right)\, ds+\int_{t_0}^t e^{\int_s^t-b(\sigma,x)d\sigma}\omega(s,x)\,ds,\\
  & \hspace{11cm} (t,x)\in [t_0,T]\times [0,1],\\
ii) \ &\omega(s,\cdot)\in \Omega(u(s,\cdot)),\ a.e.\ s\in [t_0,T],\\ 
\\
iii) \ &u(t,x)=\psi^*(t,x),\ (t,x)\in (-\infty,t_0]\times [0,1].
\end{align}
    
\end{definition}

We assume that $g:[t_0,T] \times \erre\times\erre \to \erre$  satisfies the following properties:
\begin{itemize}
\item[(g1)] for every $t \in [t_0,T]$, the map $x\mapsto g\left(t, v(x),\int_{-\infty}^0 \varphi(\theta)(x)\, d\theta\right)$ belongs to $L^2([0,1])$, for every $v\in E,\, \varphi\in \B_{\tau}$;\sm
\item[(g2)] for every $p,q\in\erre$, the function $g(\cdot, p,q)$ is (strongly) measurable;\sm
\item[(g3)]  for a.e. $t \in [t_0,T]$, the function $g(t,\cdot,\cdot)$ is continuous;\sm
\item[(g4)] there exists  $h\in L^1_+([t_0,T])$ such that 
	$
	|g(t,p,q)|\le h(t), 
	$
	for a.e. $t \in [t_0,T]$ and every $p,q\in\erre$;\sm
\item[(g5)]
there exists $q\in L^1_+([t_0,T])$ such that,  for a.e. $t \ge 0$
$$
\left\|g\left(t, v_1(\cdot), \int_{-\infty}^0 \varphi_1(\theta)(\cdot)\, d\theta\right)-g\left(t, v_2 (\cdot), \int_{-\infty}^0 \varphi_2(\theta)(\cdot)\, d\theta\right)\right\|_E
\le q(t)\left( \|v_1-v_2\|_E + \|\varphi_1-\varphi_2\|_{\B_\tau} \right),
$$
 for all $v_1,v_2\in E,\, \varphi_1,\varphi_2\in \B_{\tau}$;\sm
 \item[(g6)]
 the map $t\mapsto \|g(t,0_E,0_{\B_{\tau}})\|_E$ belongs to $L^1_+([t_0,T])$.
\end{itemize}

We consider the next properties on the multifunction $\Omega:E\to \P(E)$:
\begin{itemize}
\item[($\Omega$1)] $\Omega$ takes compact convex values;\sm

\item[($\Omega$2)] $\Omega$ is upper semicontinuous;\sm

\item[($\Omega$3)] 
$\Omega$ is compact, i.e. maps bounded sets into relatively compact sets; \sm

\item[($\Omega$4)]   there exists $R>0$ such that  \ $	\|\Omega(v)\|_E\le R(1+\|v\|_E),
	$ for every $v\in E$.	
\end{itemize}

\begin{theorem}\label{t:appl}
Suppose that the mappings $b:[t_0,T] \times [0,1]\to \erre^+$,  $g:[t_0,T] \times \erre\times\erre \to \erre$, and $\Omega:E\to \P(E)$, respectively satisfy properties (b1)-(b3), (g1)-(g6), and ($\Omega$1)-($\Omega$4). 
Further,  suppose that  the functions ${\mathcal I}_1,\dots,{\mathcal I}_m$, are bounded and continuous.

Then, the feedback control population dynamics system \eqref{p-appl} admits an optimal solutions, that is a mild solution $(u,\omega)$ (see Definition \ref{d-stc}) with $u$ 
minimizing or maximizing a cost functional associated to the system, depending if this is lower or upper semicontinuous. 
\end{theorem}

\begin{proof}
We define the next functions:
\begin{itemize}
\item 
$y:(-\infty,T]\to E$,
\begin{equation*}
\label{e-yu}
y(t)(x)=u(t,x), \, x\in [0,1];
\end{equation*}
\item 
$f:[t_0,T]\times E\times \B_{\tau}\to E$,
\begin{equation*}
%\label{e:fg}
f(t,v,\varphi)(x)=g\left(t,v(x),\int_{-\infty}^0 \varphi(\theta)(x)d\theta\right),\, x\in [0,1];
\end{equation*}
\item 
$\eta:[t_0,T]\to E$,
\begin{equation*}
\label{e:eo}
\eta(t)(x)=\omega(t,x), \,  x\in [0,1];
\end{equation*}
\item 
$\varphi^*:(-\infty,t_0]\to E$,
\[
\varphi^*(t)(x)=\psi^*(t,x),\, x\in [0,1];
\]
\item 
$I_k:\B_{\tau}\to E$,
\begin{equation*}
%\label{e:Ii}
I_i(\varphi)(x)={\mathcal I}_i\left(\int_{-\infty}^0 \varphi(\theta)(x) \, d\theta\right),\, x\in [0,1].
\end{equation*}
\end{itemize}

Thus, problem \eqref{p-appl} can be rewritten as
\begin{equation*}
\begin{cases}
y'(t)= A(t)y(t) + f\left(t,y(t),y_t\right)+ \eta(t), 
\ t \in [t_0,T]\setminus\{t_1,\dots, t_m\}, \\ 
\\
\eta(t)\in \Omega(y(t)),\  t \in [t_0,T],\\
\\
y(t)=\varphi^*(t),\, t \in (-\infty,t_0], \\
\\
y(t_k^+)=y(t_k) + I_k \left( y_{t_k} \right) ,\, k=1,\dots, m,
\end{cases}
\end{equation*} 
which in turn is a problem of type (P), if we put 
\[F(t,y(t),y_t)=f(t,y(t),y_t)+\Omega(y(t)), \  t \in [t_0,T].\]

The thesis follows from checking that the above hypotheses ensure that the assumptions of Theorem \ref{t:opt} are satisfied.
\end{proof}

\np
{\bf Acknowledgement and funding.} 

This study was partly funded by the Unione europea - Next Generation EU, Missione 4 Componente C2 -  CUP Master: J53D2300390 0006, CUP: J53D23003920 006 - Research project of MUR (italian Ministry of University and Research) PRIN 2022  “Nonlinear differential problems with applications to real phenomena” (Grant Number: 2022ZXZTN2).

The authors are members of the ``Gruppo Nazionale per l'Analisi Matematica, la Probabilità e le loro Applicazioni'' (GNAMPA) of the Istituto Nazionale di Alta Matematica (INdAM).
\md

\np
{\bf Data availability:} The manuscript has no associated data.

\np
{\bf Conflicts of interest.} The author declares no conflict of interest.
\md


\begin{thebibliography}{99} {\small

\bibitem{BZ}
M. Benchohra, M. Ziane, Impulsive Evolution Inclusions with State-Dependent Delay and Multivalued Jumps, Electronic Journal of Qualitative Theory of Differential Equations, 42 (2013), pp. 1--21.

%\bibitem{CaMaRu26}
%T. Cardinali, S. Matucci, P. Rubbioni, Stability of solutions in impulsive integro-differential equations with applications to fading memory systems, Commun. Nonlinear Sci. Numer. Simul., 154 (2026), 109588.

\bibitem{CR-2005}
T. Cardinali, P. Rubbioni, On the existence of mild solutions of semilinear evolution differential inclusions, J. Math. Anal. Appl., 308 (2005), pp. 620--635.

\bibitem{Diestel}
J. Diestel, W. M. Ruess, W. Schachermayer, Weak compactness in $L^1(\nu,X)$, Proc. Amer. Math. Soc., 118 (1993), pp. 447--453.

\bibitem{EzSt25}
K. Ezzinbi, Y. Staili, Viability for a class of partial functional differential equations with non-dense domain.
J. Math. Anal. Appl., 548, 1 (2025), Paper No. 129350, 24 pp.

\bibitem{Fan2008}
Z. Fan, Existence and continuous dependence results for nonlinear differential inclusions with infinite delay, Nonlinear Analysis, 69 (2008), pp. 2379--2392.

\bibitem{gorr02}
C. Gori, V. Obukhovskii, M. Ragni, P. Rubbioni, Existence and continuous dependence results for semilinear differential inclusions with infinite delay, Nonlinear Anal. T.M.A., 51, 5 (2002), pp. 765-782 

\bibitem{g}
L. G\'orniewicz,  Topological fixed point theory of
multivalued mappings. Second edition, Topological Fixed Point
Theory and Its Applications, 4. Springer, Dordrecht, 2006.

\bibitem{gho}
J. R. Graef, J. Henderson, A. Ouahab, Impulsive differential inclusions. A fixed point approach. De Gruyter Series in Nonlinear Analysis and Applications, 20. De Gruyter, Berlin, 2013.

\bibitem{Guedda}
L. Guedda, On the structure of the solution set of abstract inclusions with infinite delay in a Banach space, Topological Methods in Nonlinear Analysis, 48, 2 (2016), pp. 567–595

\bibitem{hk}
J.K.Hale, J.Kato, Phase space for retarded equations with infinite delay,
Funkcial. Ekvac., 21  (1978), pp. 11--41.

\bibitem{H1}
H.R. Henriquez, Periodic solutions of quasi-linear partial functional differential equations with unbounded delay, Funkcial. Ekvac., 37 (1994), pp. 329–343.

\bibitem{H2}
H.R. Henriquez, Regularity of solutions of abstract retarded functional differential equations with unbounded delay, Nonlinear Anal. TMA, 28 (1997), pp. 513–531.

\bibitem{H3}
E. Hern\'andez, H. R. Henriquez, Existence of Periodic Solutions of Partial Neutral Functional Differential Equations with Unbounded Delay, J. Math. Anal. Appl., 221 (1998), pp. 499--522.

\bibitem{Hernandez2007}
E. Hern\'andez, M. Rebello, H. R. Henríquez, Existence of solutions for impulsive partial neutral functional differential equations, J. Math. Anal. Appl., 331 (2007), pp. 1135–-1158.

\bibitem{hmn}
Y. Hino, S. Murakami, T. Naito, Functional-differential equations with infinite delay. Lecture Notes in Mathematics, 1473. Springer-Verlag, Berlin, 1991.

\bibitem{KOZ}
M. Kamenskii, V. Obukhovskii and P. Zecca, Condensing
Multivalued Maps and Semilinear Differential Inclusions in Banach
Spaces, De Gruyter Ser. Nonlinear Anal. Appl. 7, Walter de
Gruyter, Berlin - New York, 2001.

\bibitem{k}
S.G. Krein, Linear Differential Equations in Banach Spaces. American Mathematical Social, Providence (1971).

\bibitem{LWZ}
V. Lakshmikantham, L. Wem, B. Zhang, Theory of Differential Equations with Unbounded Delay, Mathematics and its Applications, Vol. 298, Kluwer, Academic Publishers, Dordrecht, 1994.

\bibitem{Leiva2023}
H. Leiva, K. Garcia, E. Lucena, Existence of solutions for semilinear retarded equations with non-instantaneous impulses, non-local conditions, and infinite delay, Open Mathematics, 21 (2023), 20230106.

\bibitem{Benchora2024}
S. Meslem, A. Salim, S. Abbas, M. Benchohra Periodic mild solutions of infinite delay integro-differential inclusions with non instantaneous impulses, Appl. Anal. Optim., 8, 1 (2024), pp. 1–-14.

\bibitem{o}
V. Obukhovskii, Semilinear functional-differential
inclusions in a Banach space and controlled parabolic systems,
Soviet J. Automat. Inform. Sci., 24, 3 (1991), pp. 71--79

\bibitem{p}
A. Pazy, Semigroups of linear operators and applications to partial differential equations.
Applied Mathematical Sciences, 44. Springer-Verlag, New York, 1983.

\bibitem{R22}
P. Rubbioni, Solvability for a Class of Integro-Differential Inclusions Subject to Impulses on the Half-Line,
Mathematics, 10, 2 (2022), 224.

\bibitem{Ru26}
P. Rubbioni,
Existence of optimal periodic strategies in a model with nonlocal spatiotemporal dispersal, J. Math. Anal. Appl. 556 (2026), 130095.

\bibitem{SNM}
R. Sakthivel, Juan J. Nieto and N. I. Mahmudov, Approximate controllability of nonlinear deterministic and stochastic systems with unbounded delay,
Taiwanese Journal of Mathematics, 14, 5 (2010)), pp. 1777--1797.

\bibitem{W}
S. Willard, General Topology, Addison Wesley Publishing Company, 1968.

}
\end{thebibliography}
\end{document}